\documentclass[11pt,reqno]{amsart}
\usepackage{xifthen,setspace}
\usepackage{bm}
\usepackage{esvect}
\usepackage{amsmath} 
\usepackage{lipsum}  
\usepackage{mathrsfs}

\usepackage{pkgfile}

\newtheorem{theorem}{Theorem}

\renewcommand{\leq}{\leqslant}
\renewcommand{\geq}{\geqslant}

\newcommand{\p}{{\mathbb{P}}}
\newcommand{\e}{{\mathbb{E}}}
\newcommand{\bs}{{\bar{X}}}

\allowdisplaybreaks

\title[Moderate Deviation and Berry-Esseen Bounds in the $p$-Spin Curie-Weiss Model]{Moderate Deviation and Berry-Esseen Bounds in the $p$-Spin Curie-Weiss Model}

\author[Mukherjee]{Somabha Mukherjee} 
\address{Department of Statistics and Data Science \\ National University of Singapore, Singapore}
\email{somabha@nus.edu.sg} 

\author[Liu]{Tianyu Liu}
\address{Department of Statistics and Data Science \\ 
	National University of Singapore, Singapore }
\email{tianyu.liu@u.nus.edu}

\author[Bhattacharya]{Bhaswar B. Bhattacharya}
\address{Department of Statistics and Data Science \\ 
	University of Pennsylvania, Philadelphia, USA }
\email{bhaswar@wharton.upenn.edu} 

\begin{document}
	
	\dedicatory{Dedicated to the memory of K.~R.~Parthasarathy}
	
	\begin{abstract}
		Limit theorems for the magnetization in the $p$-spin Curie-Weiss model, for $p \geq 3$, has been derived recently by Mukherjee et al. \cite{smfl}. In this paper, we strengthen these results by proving Cram\'er-type moderate deviation 
		theorems and Berry-Esseen bounds for the magnetization (suitably centered and scaled). In particular, we show that the rate of convergence is $O(N^{-\frac{1}{2}})$ when the magnetization has asymptotically Gaussian fluctuations, and it is $O(N^{-\frac{1}{4}})$ when the fluctuations are non-Gaussian. As an application, we derive a 
		Berry-Esseen bound for the maximum pseudolikelihood estimate of the inverse temperature in $p$-spin Curie-Weiss model with no external field, for all points in the parameter space where consistent estimation is possible. 
	\end{abstract}



	\maketitle
	
	\section{Introduction}\label{intr}

	The $p$-spin Curie-Weiss model, for $p \geq 2$, with inverse temperature $\beta > 0$ and a magnetic field $h > 0$, is a probability measure on $\{-1,1\}^N$ with mass function:
	\begin{align}\label{eq:cw}
		\p_{\beta,h,p}(\bm X) = \frac{\exp\left\{\beta N^{1-p}\sum_{1\leq i_1,\ldots,i_p\leq N} X_{i_1}\ldots X_{i_p} + h\sum_{i=1}^N X_i\right\}}{2^N Z_N(\beta,h,p)} , 
	\end{align}
	for $\bm X = ( X_1, X_2, \ldots, X_n) $. Here, 
	$Z_N(\beta,h,p)$ is the normalizing constant (also known as the partition function), which is determined  by the condition $\sum_{\bm x \in \{-1,1\}^N} \p_{\beta,h,p}(\bm x) = 1$, that is, 
	\begin{align*}
		Z_N (\beta,h,p)  = \frac{1}{2^{N}} \sum_{\bm X \in \{-1,1\}^N}  \exp \left\{ \frac{\beta}{N^{p-1}} \sum_{1 \leq i_1, i_2, \ldots, i_p \leq N} X_{i_1} X_{i_2} \cdots X_{i_p} + h \sum_{i=1}^N X_i \right \} .
	\end{align*}
	The Curie-Weiss model has been extensively studied for the case $p=2$ (see \cite{dm_gibbs_notes,ellis_book,ellis,glauber_dynamics} among several others). It corresponds to the Ising model \cite{ising} on the complete graph and is the prototypical example of a mean-field spin system with pairwise interactions. For $p \geq 3$, the $p$-spin Curie-Weiss model is a higher-order spin system where all the $p$-tuples of interactions are present (specifically, an Ising model on the complete $p$-partite hypergraph). Higher-order Ising models appear frequently in the analysis of multi-atom interactions in lattice gas models, such as the square-lattice eight-vertex model, the Ashkin-Teller model, and Suzuki's pseudo-3D anisotropic model (see \cite{ab_ferromagnetic_pspin,multispin_simulations,ferromagnetic_mean_field,pspinref1,ising_suzuki,ising_general,turban,pspinref2} and the references therein). Recently, higher-order spin systems have also found applications in statistics, for modeling peer-group effects in social networks \cite{cd_ising_II,smmpl,mlepaper}.
	%
	%

	One of the fundamental quantities of interest in an Ising model is the \textit{total magnetization} $S_N := \sum_{i=1}^N X_i$ (hereafter, simply referred to as the magnetization). 
	%
	%
	%
	%
	%
	Limit theorems for the magnetization in the 2-spin Curie-Weiss model was established in the classical work of Ellis and Newman \cite{ellis}. 
	 Berry-Esseen bounds for magnetization in 2-spin Curie-Weiss model, that is, rates of convergence in the Kolmogorov distance, were subsequently obtained by Chatterjee and Shao \cite{chatterjeeshao} and Eichelsbacher and L\"owe \cite{eichelsbacher}, using Stein's method based on exchangeable pairs. Dommers and Eichelsbacher \cite{inhomcw} later established Berry-Esseen bounds for the inhomogeneous Curie-Weiss model with external field. Similar results for the closely related Potts model was derived in \cite{eichpotts,mart1}. 
	%
	%
	
	Instead of quantifying the rate of convergence in terms of the Kolmogorov distance, one might also consider the \textit{relative error} of the tail probabilities. This is the well-understood for the sum of i.i.d. random variables through Cram\'er's classical moderate deviation theorem \cite{moderatedeviations}. Specifically, for i.i.d. random variables $Z_1, Z_2, \ldots, Z_n$ with mean $0$ and variance $1$, such that the moment generating function of $\sqrt{|Z_1|}$ is finite in a neighborhood of $0$, one has: 
	$$\frac{\p\left(\sum_{i=1}^n Z_i > x\right)}{1-\Phi(x)} = 1 + (1+x^3) O\left(\frac{1}{\sqrt{N}}\right), $$
	for $x\in [0,N^{\frac{1}{6}}]$, where $\Phi$ denotes the cumulative distribution function of the standard Gaussian (see \cite{petrov} for a proof of this result and further discussion). Cram\'er-type moderate deviation results find many applications in statistics, for instance in global and simultaneous testing, stationary processes, $L$-statistics, self-normalized sums, linear processes, among others (see, for example, \cite{Lprocess,liushao, liushao2, peligard, petrov,  shaowang, wbwu}). For the 2-spin Curie-Weiss model, Chen et. al. \cite{chenshao} derived Cram\'er-type moderate deviation theorems for the magnetization at the non-critical points using Stein's method. Moderate deviation for the critical $2$-spin Curie-Weiss model was subsequently obtained by Can and Pham \cite{vanh}, using direct arguments based on the Laplace method. 
	
	In this paper we derive moderate deviations and Berry-Esseen bounds for the magnetization in $p$-spin Curie-Weiss model, for $p \geq 3$, providing  rates of convergence for the asymptotic distributions derived in \cite{smfl}. We briefly summarize of our results below: 
	
	\begin{itemize}
		
		\item  For almost all points in the parameter space $\Theta := (0,\infty)^2$ the average magnetization is known to concentrate at a unique point and have limiting Gaussian fluctuations centered around that point (see \cite[Theorem 2.1 (1)]{smfl}). These points will be hereafter referred to as \textit{regular points}. For these points we establish a moderate deviation convergence rate of order $O(N^{-\frac{1}{2}})$. More precisely, we show that the relative error of the tail probability of the magnetization (suitably centered and scaled) and that of the Gaussian distribution at some point $x$ is bounded by $\max\{1,x^3\} O(N^{-\frac{1}{2}})$ (see Theorem \ref{thm:moderatedeviationsresult} (1)). This moderate deviation result translates to a Berry-Essen bound of rate $O(N^{-\frac{1}{2}})$ at the regular points (see Theorem \ref{thm:berryesseen} (1)).
		
		\item There is a 1-dimensional curve in the parameter space (hereafter referred to as \textit{critical points}), where the average magnetization concentrates at more than one (either two or three) points, and has limiting Gaussian fluctuations centered around each of the points, when conditioned to lie in their respective neighborhoods (see \cite[Theorem 2.1 (2)]{smfl}). For these points, similar to the regular points, we show that the moderate deviation relative error at a point $x$ is bounded by $\max\{1,x^3\} O(N^{-\frac{1}{2}})$  (see Theorem \ref{thm:moderatedeviationsresult} (2)) and the Berry-Esseen convergence rate is of order $O(N^{-\frac{1}{2}})$ (see Theorem \ref{thm:berryesseen} (2)). 
		
		\item There are also one or two points in the parameter space, depending on whether $p \geq 3$ is odd or even, respectively (hereafter referred to as \textit{special points}), where the magnetization has a non-Gaussian limiting distribution (see \cite[Theorem 2.1 (3)]{smfl}). For these points we show that the moderate deviation relative error at a point $x$ is bounded by $\max\{1,x^5\} O(N^{-\frac{1}{4}})$ (see Theorem \ref{thm:moderatedeviationsresult} (3)). This also gives a Berry-Essen bound of rate $O(N^{-\frac{1}{4}})$ at the special points (see Theorem \ref{thm:berryesseen} (3)).
		
	\end{itemize}

	The proofs for the moderate deviations result for the regular and critical points use Stein's method based on exchangeable pairs developed in \cite{chenshao}. For the special points we use Laplace's method for approximating the normalizing constant, similar to that in \cite{vanh} for the 2-spin case, although the arguments are more delicate for $p \geq 3$ due the presence of additional degeneracies (see Remark \ref{remark:moderatedeviationsrate}). The Berry-Esseen bounds follow from the moderate deviations results and the boundedness of the moments of the magnetization (suitably centered and scaled). 
	%
	%
	
	We conclude with an application in statistical inference, by deriving Berry-Esseen bounds for the maximum pseudolikelihood (MPL) estimate of $\beta$ given a single sample $\bm X$ from the model \eqref{eq:cw}  (assuming $h=0$). In particular, we prove a $O(N^{-\frac{1}{2}})$ rate of convergence (up to logarithmic factors) for the MPL estimate for all values of $\beta$ where consistent estimation is possible (see Theorem \ref{thm:berryesseenestimate}). 
	%
	%
	
	\subsection{Organization:} The rest of the paper is organized as follows. In Section \ref{Prel}, we recall the results on the asymptotic distribution of the magnetization from \cite{smfl}. 
	In Section \ref{sec:statements} we state the main results of the paper. The proofs of the main results are given Section \ref{sec:proofmainres}. We discuss some possible future directions in Section \ref{sec:extensions}. 
	

	\section{Fluctuations of the Magnetization}\label{Prel}

	In this section, we will recall from Mukherjee et al. \cite{smfl} the results on the asymptotic distribution of the total magnetization $S= S_N := \sum_{i=1}^N X_i$ in the $p$-spin Curie-Weiss model \eqref{eq:cw}.  
	The limiting distribution of $S$ depends on location of the parameters $(\beta,h)$ in the parameter space $\Theta := (0,\infty)^2$. In particular, the maximizer(s) of the following function will play an important role: 
	\begin{align}\label{Hdefn}
		H(x) = H_{\beta,h,p}(x) := \beta x^p + h x - I(x)\quad(x\in [-1,1])
	\end{align}
	where $I(x) := \frac{1}{2}\left\{(1+x)\log (1+x) + (1-x)\log (1-x)\right\}$ denotes the binary entropy function. It was shown in \cite{smfl} that the function $H$ can have one, two, or three global maximizers in the open interval $(-1, 1)$. This leads to the following partition of the parameter space $\Theta$:\footnote{For a smooth function $f: [-1, 1] \rightarrow \R$ and $x \in (-1, 1)$, the first and second derivatives of $f$ at the point $x$ will be denoted by $f'(x)$ and $f''(x)$, respectively. More generally, for $s \geq 3$, the $s$-th order derivative of $f$ at the point $x$ will be denoted by $f^{(s)}(x)$. } 
	
	\begin{enumerate} 
		
		\item The point $(\beta,h)$ is said to be $p$-{\it regular}, if the function $H_{\beta,h,p}$ has a unique global maximizer $m_* = m_*(\beta,h,p) \in (-1,1)$ and $H_{\beta,h,p}''(m_*) < 0$.\footnote{A point $m \in (-1, 1)$ is a global maximizer of $H$ if $H(m) > H(x)$, for all $x\in [-1,1]\setminus \{m\}$.} Denote the set of all $p$-regular points in $\Theta$ by $\cR_p$.

		\item The point $(\beta,h)$ is said to be $p$-{\it special}, if $H_{\beta,h,p}$ has a unique global maximizer $m_* = m_*(\beta,h,p) \in (-1,1)$ and $H_{\beta,h,p}''(m_*) = 0$.  Denote the set of all $p$-special points in $\Theta$ by $\cS_p$.

		\item The point $(\beta,h)$ is said to be $p$-{\it critical}, if $H_{\beta,h,p}$ has more than one global maximizer. Denote the set of all $p$-critical points in $\Theta$ by $\cC_p$.  
		
	\end{enumerate}

	Note that the three cases above form a disjoint partition of the parameter space $\Theta$ (observe that $H_{\beta,h,p}''$ is non-positive at the global maximizer, by the higher-order derivative test). 
	Furthermore, in \cite[Lemma B.2]{smfl} it was shown that 
	when $p \geq 3$, there is only $p$-special point in $\Theta$ and when $p \geq 4$ even, there are two $p$-special points, which are symmetric about $h=0$. Also, it was shown in \cite[Lemma B.3]{smfl} that the set of points in $\cC_p$ form a continuous $1$-dimensional curve in the parameter space $\Theta$. For a schematic of the partition of the parameter space, see Figures 6 and 7 in \cite{smfl}. 	
	The following result from \cite{smfl} gives the asymptotic distribution of the total magnetization:

	\begin{theorem}[Theorem 2.1 in \cite{smfl}]
		Fix $p \geq 3$ and $(\beta, h) \in \Theta$, and suppose $\bm X \sim \mathbb{P}_{\beta,h,p} $. Then the following hold: 
		
		\begin{itemize}
			
			\item[$(1)$] Suppose that $(\beta, h) \in \cR_p$, and denote the unique maximizer of $H$ by $m_*= m_*(\beta,h,p)$. Then, as $N \rightarrow \infty$,
			\begin{align}\label{eq:meanclt_I}
				\frac{S_N - Nm_*}{\sqrt{N}}\xrightarrow{D} N\left(0,-\frac{1}{H''(m_*)}\right).
			\end{align}

			\item[$(2)$] Suppose $(\beta, h) \in \cC_p$,  and denote the $K$ maximizers  of $H$ by $m_1:=m_1(\beta,h,p)< \ldots < m_K:=m_K(\beta,h,p)$. Then, $K \in \{2, 3\}$, and as $N \rightarrow \infty$, 
			\begin{align*}
				\frac{S_N}{N} \xrightarrow{D} \sum_{k=1}^K p_k \delta_{m_k}, 
			\end{align*}
			where for each $1\leq k\leq K$,
			\begin{align*}
				p_k := \frac{\left[(m_k^2-1)H''(m_k)\right]^{-1/2}}{\sum_{i=1}^K \left[(m_i^2-1)H''(m_i)\right]^{-1/2}}.
			\end{align*} 
			Moreover, if $A \subseteq [-1,1]$ is an interval containing $m_k$ in its interior for some $1\leq k \leq K$, such that $H(m_k) > H(x)$ for all $x\in A\setminus \{m_k\}$, then 
			\begin{align}\label{eq:meanclt_II}
				\frac{S_N - Nm_k}{\sqrt{N}}\Big\vert \{S_N/N \in A\}  \xrightarrow{D} N\left(0,-\frac{1}{H''(m_k)}\right).\end{align}

			\item[$(3)$] Suppose that $(\beta,h) \in \cS_p$, and denote the unique maximizer of $H$ by $m_*= m_*(\beta,h,p)$. Then, as $N \rightarrow \infty$, 
			\begin{align}\label{eq:meanclt_III}
				\frac{S_N - Nm_*}{N^\frac{3}{4}}\ \xrightarrow{D} F, 
			\end{align}
			where the density of $F$ with respect to the Lebesgue measure is given by 
			\begin{align}\label{eq:meanclt_g}
				\mathrm dF(x) = \frac{2}{\Gamma(\tfrac{1}{4})}\left(-\frac{H^{(4)}(m_*)}{24}\right)^{\frac{1}{4}}\exp\left(\frac{H^{(4)}(m_*)}{24} x^4\right)\mathrm d x,
			\end{align}
			with $H^{(4)}$ denoting the fourth derivative of the function $H$. 
		\end{itemize}
	\end{theorem}

	\section{Statements of the Main Results}\label{sec:statements} 
	
	In this paper, we derive the rates of the convergence (in the Kolmogorov distance) for the results in \eqref{eq:meanclt_I}, \eqref{eq:meanclt_II} and \eqref{eq:meanclt_III} by proving Cram\'er-type moderate deviations for the random variable 
	\begin{align}\label{eq:WN}
		W_N := \frac{S_N-Nm}{\alpha_N}
	\end{align}
	where $\alpha_N = \sqrt{N}$ if $(\beta,h)\in \cR_p \bigcup \cC_p$ and $\alpha_N = N^{\frac{3}{4}}$ if $(\beta,h)\in \cS_p$, while $m$ denotes the unique maximizer of $H$ if $(\beta,h) \in \cR_p\bigcup \cS_p$, and one of the multiple global maximizers of $H$ if $(\beta,h)\in \cC_p$. We formally state the moderate deviations result in Section \ref{sec:moderatedeviations}. The Berry-Esseen bounds are stated in Section \ref{sec:convergence}. In Section \ref{sec:estimationpl} we use this result to obtain Berry-Esseen bounds for the MPL estimate of $\beta$.


	\subsection{Moderate Deviations for the Magnetization }
	\label{sec:moderatedeviations}

	Throughout $\Phi$ will denote the standard normal distribution function. Then we have the following moderate deviations result for the total magnetization  $S_N := \sum_{i=1}^N X_i$. 
	
	
	\begin{theorem}\label{thm:moderatedeviationsresult}
		Suppose that $\bm X$ is a sample from the model \eqref{eq:cw} with $p\geq 3$. Then, there exists a constant $C>0$ such that
		\begin{enumerate}
			\item If $(\beta,h)\in \cR_p$, then for $r\in \{0,1\}$, 
			$$\frac{\p\left((-1)^r \frac{S_N-Nm_*}{\sqrt{-N/H''(m_*)}}>x\right)}{1-\Phi(x)} = 1 + (1+ x^3) O\left(\frac{1}{\sqrt{N}}\right),$$
			for all $x\in [0,CN^{\frac{1}{6}}]$, where $m_*$ is the unique global maximizer of the function $H$ defined in \eqref{Hdefn}.
			\vspace{0.2cm}
			
			\item  If $(\beta,h)\in \cC_p$, then for $r\in \{0,1\}$, 
			$$\frac{\p\left((-1)^r \frac{S_N-Nm}{\sqrt{-N/H''(m)}}>x\Bigg|~\frac{S_N}{N}\in A\right)}{1-\Phi(x)} = 1 + (1+ x^3) O\left(\frac{1}{\sqrt{N}}\right),$$
			for all $x\in [0,CN^{\frac{1}{6}}]$, where $m$ is any global maximizer of $H$, and $A$ is a neighborhood of $m$ whose closure excludes any other maximizer of $H$.
			\vspace{0.2cm}
			
			\item If $(\beta,h)\in \cS_p$, then for $r\in \{0,1\}$, 
			$$\frac{\p\left((-1)^r\frac{S_N-Nm_*}{N^{\frac{3}{4}}}>x\right)}{1-F(x)} = 1 + (1+ x^{5}) O\left(\frac{1}{N^{1/4}}\right),$$
			for all $x\in [0,CN^{\frac{1}{20}}]$, where $m_*$ is the unique global maximizer of $H$, and the density of $F$ with respect to the Lebesgue measure is given by 
			\begin{align*}
				\mathrm dF(x) = \frac{2}{\Gamma(\tfrac{1}{4})}\left(-\frac{H^{(4)}(m_*)}{24}\right)^{\frac{1}{4}}\exp\left(\frac{H^{(4)}(m_*)}{24} x^4\right)\mathrm d x,
			\end{align*}
			with $H^{(4)}$ denoting the fourth derivative of the function $H$. 
		\end{enumerate}
	\end{theorem}
	
	The proof of Theorem \ref{thm:moderatedeviationsresult} is given in Section \ref{sec:moderatedeviationspf}. For parts (1) and (2) we apply the technique for proving moderate deviations develop in  \cite{chenshao} based on Stein's method. The proof of (3) is a more bare-hands approach based on the Laplace method-type arguments, similar to that in \cite{vanh} where moderate deviations for the case $p=2$ was obtained.  
	
	\begin{remark}\label{remark:moderatedeviationsrate}
		{\em One interesting difference in the rates for $p \geq 3$, compared to the case $p=2$, is that for special points the rate of convergence is $O(N^{-\frac{1}{4}})$ for $p \geq 3$, whereas for $p=2$ the rate of convergence is $O(\frac{1}{\sqrt N})$ (see \cite[Corollary 1.5]{vanh}).  This is because for $p = 2$ at the special point, the fifth-derivative $H^{(5)}(m_*) = 0$, unlike that for $p \geq 3$ at the special points, where $H^{(5)}(m_*) \ne 0$. This non-zero $H^{(5)}$ term determines the rate of convergence of $N^{-1/4}$ for $p\ge 3$. As a result, more refined calculations are required $p \geq 3$ at the special points. } 
	\end{remark} 
	
	\subsection{Berry-Esseen Bounds for the Magnetization}
	\label{sec:convergence}
	
	The moderate deviation results in Theorem \ref{thm:moderatedeviationsresult} coupled with the convergence of moments of $W_N$ (recall \eqref{eq:WN}) imply the following Berry-Esseen type bounds. The proof is given in Section \ref{sec:convergencepf}. 
	
	\begin{theorem}\label{thm:berryesseen}
		Suppose that $\bm X$ is a sample from the model \eqref{eq:cw} with $p\geq 3$. Then, 
		\begin{enumerate}
			\item If $(\beta,h)\in \cR_p$, then 
			$$\sup _{x \in \mathbb{R}}~\left|\p\left(\frac{S_N-Nm_*}{\sqrt{-N/H''(m_*)}}\leq x\right)-\Phi(x)\right| = O\left(\frac{1}{\sqrt{N}}\right), $$ 
			where $m_*$ is the unique global maximizer of the function $H$. 
			\vspace{0.2cm}
			
			\item  If $(\beta,h)\in \cC_p$, then for all global maximizers $m$ of $H$, 
			$$\sup _{x\in \mathbb{R}} ~\left|\p\left(\frac{S_N-Nm}{\sqrt{-N/H''(m)}}\leq x\Bigg|~\frac{S_N}{N}\in A\right) - \Phi(x)\right| = O\left(\frac{1}{\sqrt{N}}\right),$$ with $A$ being a neighborhood of $m$ whose closure excludes any other maximizer of $H$.
			\vspace{0.2cm}
			\item  If $(\beta,h)\in \cS_p$, then 
			$$\sup _{x>0} ~\left|\p\left(\frac{S_N-Nm_*}{N^{\frac{3}{4}}}\leq x \right) - F(x)\right| = O\left(\frac{1}{N^{1/4}}\right),$$  where $m_*$ is the unique global maximizer of the function $H$.
		\end{enumerate}
	\end{theorem}
	
	\begin{remark} {\em 
			Recently, Can and R\"ollin \cite{rollin} studied a class of  mean-field spin models where the Gibbs measure of each configuration depends only on its magnetization, which includes as a special case the $p$-spin Curie-Weiss model. They obtained rates of convergence of the magnetization in the Wasserstein distance in this general model, which, in the specific case of the $p$-spin Curie-Weiss model, for $p \geq 3$, imply rates of convergence of $O(N^{-\frac{1}{2}})$ for the regular and the critical points and of $O(N^{-\frac{1}{4}})$ at the critical points (see \cite[Theorem 5.1]{rollin}). In contrast, Theorem \ref{thm:berryesseen} gives Berry-Esseen bounds of the same orders, which do not directly follow from the Wasserstein bounds in \cite{rollin}. 
				Moreover, the moderate deviation bounds in Theorem \ref{thm:moderatedeviationsresult}, which quantifies the error in  a stronger sense than the Kolmogorov distance, are, to the best of our knowledge, the first such results for $p$-spin Curie-Weiss model, for $p \geq 3$. 
		 } 
	\end{remark}
	
	\subsection{Berry-Esseen Bounds for the Maximum Pseudolikelihood Estimate of $\beta$}\label{sec:estimationpl}
	
	The problem of parameter estimation in Ising models has been extensively studied \cite{BM16,chatterjee,comets_exp,comets,cd_ising_estimation,cd_ising_I,pg_sm,gidas,guyon,rm_sm,
			ising_testing,hdcltcov,smmpl, mlepaper,discrete_mrf_pickard}. Here, we consider estimating the parameter $\beta$ (assuming $h=0$) given a single sample $\bm X = (X_1, X_2, \ldots, X_n)$ from the model \eqref{eq:cw}. Maximum likelihood estimation in this model is often inexplicit due to the presence of the intractable normalizing constant $Z_N$. An alternative computationally feasible strategy is the \textit{maximum pseudolikelihood} (MPL) method \cite{besag_lattice,besag_nl}, which leverages the fact that the conditional distributions $\p_{\beta,0,p}(X_i|(X_j)_{j\ne i})$ have simple explicit forms, for $1\leq i\leq N$. In particular, the MPL estimate of $\beta$ in the model \eqref{eq:cw} (with $h=0$) is obtained as follows: 
	$$\hat{\beta} := \arg \max_{\beta\in \mathbb{R}} L(\beta|\bm X) , $$
	where  
	$$L(\beta|\bm X) := \prod_{i=1}^N \p_{\beta,0,p}\left(X_i|(X_j)_{j\ne i}\right)$$ is the \textit{pseudolikelihood} function. 
	
	From results in \cite{smmpl} we know that there exists $\beta^*(p) >0$ such that for for $\beta>\beta^*(p)$, the MPL estimator $\hat{\beta}$ is $\sqrt{N}$-consistent, whereas for $\beta < \beta^*(p)$, no consistent estimate of $\beta$ exists. The threshold $\beta^*(p)$ is given by:
	$$\beta^*(p) := \inf \left\{\beta \geq 0:~ \sup_{x\in [0,1]} H_{\beta,0,p}(x) > 0\right\}.$$
	In fact,  \cite[Theorem 2.14]{smmpl} shows that for all $\beta > \beta^*(p)$, 
	$$\sqrt{N}(\hat{\beta}-\beta) \xrightarrow{D} N\left(0, -\frac{H_{\beta,0,p}''(m_*)}{p^2 m_*^{2p-2}}\right) , $$ where $m_*$ is the unique positive global maximizer of $H_{\beta,0,p}$. The next result establishes a Berry-Esseen bound for this weak convergence. The proof is given in Section \ref{sec:convergencepf}. 
	
	\begin{theorem}\label{thm:berryesseenestimate}
		Suppose $\hat \beta$ is the MPL estimate of $\beta$ in the model $\p_{\beta,0,p}$, with $p\geq 3$. Then for $\beta > \beta^*(p)$, 
		$$\sup_{x\in \mathbb{R}}~\left|\p\left(\sqrt{N}(\hat{\beta}-\beta) \leq x\right) - \p\left(N\left(0, -\frac{H_{\beta,0,p}''(m_*)}{p^2 m_*^{2p-2}}\right) \leq x\right)\right| = O\left(\frac{\log N}{\sqrt{N}}\right) . $$
	\end{theorem}

	\section{Proofs of the Main Results}\label{sec:proofmainres}
	
	In this section, we prove the results stated in Section \ref{sec:statements}. We start by proving the moderate deviation results.
	
	\subsection{Proof of Theorem \ref{thm:moderatedeviationsresult}}
	\label{sec:moderatedeviationspf}
	
	We will prove parts (1) and (2) together. For this we will use the following result, which is a direct consequence of the results in \cite{chenshao} (see Theorem 3.1 and Remark 3.2 in \cite{chenshao}). 
	
	\begin{theorem} \cite{chenshao} \label{thm:moderatedeviations}
		Suppose that $(Y,Y')$ is an exchangeable pair of random variables satisfying:
		\begin{align}\label{eq:Y} 
			\e(Y-Y'|Y) = \lambda(Y-\e(R|Y))
		\end{align} 
		for some random variable $R$ and some constant $\lambda \in (0,1)$. Let $\Delta := Y'-Y$, $D:= \frac{\Delta^2}{2\lambda}$ and assume that:
		\begin{enumerate}
			\item $|\Delta|\leq \delta$ for some $\delta>0$,
			
			\item There exists $\delta_1>0$ such that $|\e(D|Y)-1| \leq \delta_1(1+|Y|)$,
			
			\item There exist  $\delta_2>0$ and $\alpha <1$ such that $|\e(R|Y)| \leq \delta_2(1+Y^2)$ and $\delta_2|Y|\leq \alpha$,
			
			\item There exists  $\theta\geq 1$ such that $\e(D|Y) \leq \theta$.
		\end{enumerate}
		Then,  for $0 \leq x \leq \theta^{-1}\min\{\delta^{-\frac{1}{3}},\delta_1^{-\frac{1}{3}},\delta_2^{-\frac{1}{3}}\}$,    
			\begin{align}\label{eq:moderatedeviationsrate}
				\frac{\p(Y>x)}{1-\Phi(x)} = 1+ C_\alpha \theta^3(1+x^3)(\delta+\delta_1+\delta_2) , 
			\end{align}
			where $C_\alpha > 0$ is a constant depending only on $\alpha$. \qed
	\end{theorem}

	Returning to the main proof, assume, to begin with, $r=0$. Let us define:
	\begin{align}\label{eq:RCWN}
		\tilde{W}_N := 
		\left\{
		\begin{array}{cc}
			\sqrt{-H''(m)} W_N & \text{if} ~(\beta,h)\in \cR_p   \\ 
			\sqrt{-H''(m)} W_N \Big|~\{S_N/N\in A\} & \textrm{if}~(\beta,h)\in \cC_p    
		\end{array} ,
		\right.
	\end{align}
	where $W_N := \frac{1}{\sqrt N}(S_N-Nm)$, with $m$ denoting a global maximizer of $H$, and $A$ denoting a neighborhood of $m$ whose closure excludes any other maximizer of $H$. It follows from the proofs of Lemmas 3.1 and 3.3 in \cite{mlepaper} that for every constant $K>0$, there exists a constant $L = L_K>0$ such that 
	\begin{equation*}
		\p\left(|\tilde{W}_N| > K\sqrt{N}\right) \leq e^{-LN} . 
	\end{equation*}
	Therefore, for every constant $K>0$, 
	\begin{eqnarray*}
		\frac{\p(\tilde{W}_N > x)}{1-\Phi(x)} &=& 
		\frac{\p\left(\tilde{W}_N > x \Big|~|\tilde{W}_N|\leq K\sqrt{N}\right) (1-O(e^{-LN}))}{1-\Phi(x)} + \frac{O(e^{-LN)}}{1-\Phi(x)}\\&=&  \frac{\p\left(\tilde{W}_N > x \Big|~|\tilde{W}_N|\leq K\sqrt{N}\right)}{1-\Phi(x)} + \frac{O(e^{-LN)}}{1-\Phi(x)}.
	\end{eqnarray*}
	Hence, in order to prove Theorem \ref{thm:moderatedeviationsresult}, it suffices to prove the same theorem with $\tilde{W}_N$ replaced by $T = T_N := \tilde{W}_N|\{|\tilde{W}_N| \leq K\sqrt{N}\}$ for some constant $K>0$ to be determined later. Note that the condition $\tilde{W}_N \leq K\sqrt{N}$ is equivalent to the inequality $$\left|\frac{S_N}{N} - m\right| \leq \frac{K}{\sqrt{-H''(m)}} . $$ 
	Hence, for $K$ sufficiently small, we may assume that the conditioned state space of the spin configuration $\bm X$ is such that $S_N/N \in [a,b]$, where $(a,b)$ is an interval containing exactly one maximizer $m$ of $H$. Throughout the rest of the proof, let us denote $\sigma^2 := -N/H''(m)$ and with $\bs := S/N$.
	
	Let $I$ be denote a uniformly distributed random variable on the set $\{1,\ldots,N\}$ independent of $\bm X$, and given $I=i$, let $\bm X'$ denote the random vector obtained from $\bm X$ by replacing $X_i$ with $X_i'$ generated from the conditional distribution of $X_i$ given $(X_j)_{j\ne i}$. Let  $T'$ denote the version of the statistic $T$ evaluated at $\bm X'$, i.e. $$T' = T + \frac{X_I' - X_I}{\sigma}~.$$ 
	Then, $(T,T')$ is an exchangeable pair (see \cite{smmpl}). Now, denote the events 
		$$\cE_a= \{S-2 \geq aN\} \quad \cE_b = \{S+2 \leq bN\}, $$
		and define 
		$$U := \frac{2}{N\sigma}\cdot  \frac{e^{-\beta p  \bs^{p-1} - h}}{2\cosh(\beta p  \bs^{p-1} + h)}\cdot \frac{N}{2}(1+\bs) \bm 1\{\cE_a\} \text{ and } V := \frac{2}{N\sigma}\cdot  \frac{e^{\beta p  \bs^{p-1} + h}}{2\cosh(\beta p  \bs^{p-1} + h)}\cdot \frac{N}{2}(1-\bs)\bm 1\{\cE_b\}.$$
		Then we have: 
		\begin{align}\label{eq:TX}
			\e(T-T'|\bm X) &= \frac{1}{\sigma} \e(X_I - X_I'|\bm X) \nonumber \\
			& = \frac{2}{N\sigma} \sum_{i=1}^N \left[\bm{1}(X_i=1)\p(X_i'=-1|\bm X) - \bm{1}(X_i=-1)\p(X_i'=1|\bm X)\right] \nonumber \\
			& = U-V . 
	\end{align} 
	%
	%
	Note that $$\frac{N}{2}(1+\bs) = \frac{1}{2}(\sigma T + Nm +N)\quad\text{and}\quad \frac{N}{2}(1-\bs) = \frac{1}{2}(N-\sigma T - Nm).$$
	Denote 
	$$A(T) := \frac{\exp\left(-\beta p \left(\frac{\sigma T}{n} + m\right)^{p-1}-h\right)}{2\cosh\left(\beta p \left(\frac{\sigma T}{N} + m\right)^{p-1}+h\right)} \quad\text{and}\quad B(T) := \frac{\exp\left(\beta p \left(\frac{\sigma T}{N} + m\right)^{p-1}+h\right)}{2\cosh\left(\beta p \left(\frac{\sigma T}{n} + m\right)^{p-1}+h\right)}~.$$
	 Then \eqref{eq:TX} can we expressed as,  
		\begin{eqnarray*}
			\e(T-T'|\bm X) &=& \frac{2}{\sigma}\cdot \frac{\sigma T + Nm + N}{2N}\cdot A(T) \bm 1\{\cE_a\} - \frac{2}{\sigma}\cdot \frac{N-\sigma T- Nm}{2N}\cdot B(T) \bm 1\{\cE_b\}\\&=& (A(T) + B(T))\left(\frac{T}{N} + \frac{m}{\sigma}\right) + \frac{1}{\sigma} (A(T)-B(T)) + Q, 
			\\&=& \left(\frac{T}{N} + \frac{m}{\sigma}\right) -\frac{1}{\sigma} \tanh\left(\beta p \left(\frac{\sigma T}{N} + m\right)^{p-1} + h\right) + Q , 
		\end{eqnarray*} 
		where 
		$$ Q: = -\frac{\sigma T + Nm + N}{N\sigma} A(T) \bm 1\{\cE_a^c\} + \frac{N-\sigma T - Nm}{N\sigma} B(T) \bm 1\{\cE_b^c\}.$$ 
	 By a $2$-term Taylor expansion of the function $$f(x) := \tanh\left(\beta p(x+m)^{p-1} + h\right)$$ around $0$, we have:
		\begin{align*}
			& \tanh\left(\beta p \left(\frac{\sigma T}{N} + m\right)^{p-1} + h\right) \nonumber \\ 
			& = \tanh(\beta p m^{p-1} + h) + \frac{\sigma T}{N} \tanh'(\beta p m^{p-1} + h) \beta p(p-1)m^{p-2} + \frac{\sigma^2 T^2}{2N^2} \tilde Q , 
		\end{align*}
		where 
		$$\tilde Q := \tanh''(\beta p(m+\xi)^{p-1}+h)\left(\beta p(p-1)(m+\xi)^{p-2}\right)^2 + \tanh'(\beta p (m+\xi)^{p-1} + h)\beta p(p-1)(p-2)(m+\xi)^{p-3},$$
		for some $\xi \in (0,\sigma T/N)$. Hence, noting that $\tanh(\beta p m^{p-1}+h) = m$, gives, 
		\begin{align}\label{eq:stein31}
			\e(T-T'|\bm X) = \lambda(T-R)
		\end{align}
		where 
		\begin{align}
			\label{eq:lambda} 
			\lambda & := \frac{1}{N}\left(1-\tanh'(\beta p m^{p-1}+h) \beta p(p-1)m^{p-2}\right) , \\ 
			\label{eq:R} 
			R & := \frac{\sigma T^2}{2N^2\lambda} \tilde Q + \frac{S + N}{N\sigma\lambda} A(T) \bm 1\{\cE_a^c\} - \frac{N-S}{N\sigma\lambda} B(T) \bm 1\{\cE_b^c\} . 
	\end{align}  
	Next, differentiating both sides of the identity $H(x) = \beta x^p + hx-I(x)$ gives, 
	$$\tanh^{-1}(x) = \beta p x^{p-1} + h - H'(x) \quad \implies\quad x = \tanh\left(\beta p x^{p-1} +h - H'(x)\right).$$ Differentiating the right hand side identity further, 
	$$\tanh'\left(\beta p x^{p-1} + h - H'(x)\right) \left(\beta p(p-1) x^{p-2} - H''(x)\right) = 1.$$
	Now, putting $x=m$ in the above identity gives, 
	$$\tanh'(\beta p m^{p-1} + h)= \frac{1}{\beta p(p-1) m^{p-2} - H''(m)}.$$ Hence, 
	$$\lambda = \frac{1}{N}\cdot \frac{H''(m)}{H''(m) - \beta p(p-1)m^{p-2}} = \frac{1-\beta p(p-1)m^{p-2}(1-m^2)}{N} \in (0,1)~,$$
	since $H''(m)<0$.  This establishes condition \eqref{eq:Y} in Theorem \ref{thm:moderatedeviations} with $\lambda$ and $R$ as in \eqref{eq:lambda} and \eqref{eq:R}, respectively (recall \eqref{eq:stein31}).

	Also, note that 
		\begin{align}
			& \e\left[(T-T')^2\Big| \bm X\right] \nonumber \\ 
			&= \frac{1}{\sigma^2}\e\left[(X_I-X_I')^2\Big| \bm X\right] \nonumber \\
			&= \frac{4}{N\sigma^2} \sum_{i=1}^N \left[\bm{1}(X_i=1)\p(X_i'=-1|\bm X) - \bm{1}(X_i=-1)\p(X_i'=1|\bm X)\right] \nonumber \\
			&= \frac{4}{\sigma^2}\cdot \frac{\sigma T + Nm + N}{2N}\cdot A(T) \bm 1\{\cE_a\} + \frac{4}{\sigma^2}\cdot \frac{N-\sigma T- Nm}{2N}\cdot B(T) \bm 1\{\cE_b\} \nonumber \\
			&= \frac{2}{\sigma^2} (A(T)+B(T)) + \frac{4}{\sigma^2} \cdot \frac{\sigma T + Nm}{2N} (A(T)-B(T)) + Q_2 \nonumber \\ 
			&= \frac{2}{\sigma^2} - \frac{2}{\sigma^2}\cdot \frac{\sigma T + N m}{N} \tanh\left(\beta p \left(\frac{\sigma T}{N} + m\right)^{p-1} + h\right) +  O\left(\frac{\bm{1}(\cE_a^c \bigcup \cE_b^c)}{\sigma^2}\right) \label{eq:TXAB}
		\end{align} 
		where 
		$$Q_2 := - \frac{4}{\sigma^2}\cdot \frac{\sigma T + Nm + N}{2N}\cdot A(T) \bm{1}\{\cE_a^c\} - \frac{4}{\sigma^2}\cdot \frac{N-\sigma T- Nm}{2N}\cdot B(T) \bm{1}\{\cE_b^c\}.$$
		Note that the last equality in \eqref{eq:TXAB} follows from the fact that $T = O(\sqrt N)$. Now, by a $1$-term Taylor expansion of the function $f(x) = \tanh(\beta p (x+m)^{p-1}+h)$ around $0$ and using the identity $m = \tanh(\beta pm^{p-1}+h)$, gives 
		$$ \tanh\left(\beta p \left(\frac{\sigma T}{N} + m\right)^{p-1} + h\right) = m + \frac{\sigma T}{N} \tanh'(\beta p (m+\xi)^{p-1} + h)\beta p(p-1)m^{p-2} = m + \frac{\sigma T}{N} O(1).$$ for some $\xi \in (0,\sigma T/N)$.  Combining all these gives, 
	\begin{eqnarray*}
		\e\left[(T-T')^2\Big| \bm X\right] &=& \frac{2}{\sigma^2}(1-m^2) + O(1)\frac{T}{N\sigma} + \frac{2T^2}{N^2} O(1)  +O\left(\frac{\bm{1}(\cE_a^c \bigcup \cE_b^c)}{\sigma^2}\right)\\&=& \frac{2}{\sigma^2}(1-m^2) + O(1)\frac{T}{N\sigma}  + O\left(\frac{\bm{1}(\cE_a^c \bigcup \cE_b^c)}{\sigma^2}\right).
	\end{eqnarray*}
	
	We will now verify the hypotheses of Theorem \ref{thm:moderatedeviations}. To begin with, note that $\Delta  := T'-T$ satisfies $|\Delta| \leq \frac{2}{\sigma} = O(\frac{1}{\sqrt N})$. Hence, trivially condition (1) of Theorem \ref{thm:moderatedeviations}  holds with $\delta = O(\frac{1}{\sqrt N})$. Next, note that 
	$$\frac{2}{\sigma^2}(1-m^2) = -\frac{2(1-m^2)H''(m)}{N} = -\frac{2(\beta p (p-1)m^{p-2} (1-m^2) - 1)}{N} = 2\lambda.$$
	Hence, we have the following with $D:=\frac{\Delta^2}{2\lambda}$, 
	\begin{eqnarray*}
		|\e(D|T)-1| &=& \left|O(1) \left(\frac{T}{2\lambda N\sigma}\right) + O\left(\frac{\bm{1}(\cE_a^c \bigcup \cE_b^c)}{2\lambda \sigma^2}\right)\right| \\&=&  O\left(\frac{1}{\sqrt{N}}\right) |T| + O\left(\bm{1}(\cE_a^c \bigcup \cE_b^c)\right)
	\end{eqnarray*}
	Note that since $m\in (a,b)$, there exists a constant $\varepsilon > 0$, such that $$\{\cE_a^c \bigcup \cE_b^c\} \subset \left\{\left|\frac{S}{N}-m\right| > \varepsilon\right\}$$ and the right hand side event becomes empty for $K$ small enough. Hence, we conclude that $$|\e(D|T)-1|= O\left(\frac{1}{\sqrt{N}}\right) (1+|T|),$$ that is, condition (2) of Theorem \ref{thm:moderatedeviations} is satisfied with $\delta_1 = O(\frac{1}{\sqrt N})$. Further, since $|T| = O(\sqrt N)$, it also follows that $|\e(D|T)| = O(1)$ and hence, condition (4) of Theorem \ref{thm:moderatedeviations} is satisfied with $\theta = O(1)$. Finally, note that for $K$ small enough, 
	$$R = \frac{\sigma T^2}{2N^2\lambda} O(1) \leq \frac{B T^2}{\sqrt{N}},$$ for some constant $B$. Now we can choose $K <1/(2B)$ to ensure that $B\frac{1}{\sqrt N}|T| < \frac{1}{2}$, which shows that condition (3) of Theorem \ref{thm:moderatedeviations} is satisfied with $\delta_2 = O(\frac{1}{\sqrt N})$. Now, since all the conditions of Theorem \ref{thm:moderatedeviations} are satisfies, invoking \eqref{eq:moderatedeviationsrate} completes the proof of parts (1) and (2) of Theorem \ref{thm:moderatedeviationsresult} for the case $r=0$.

	If $r=1$, then we can apply the same arguments for the random variables $-\tilde{W}_N$ and $-T_N$. In particular, note that the hypotheses of Theorem \ref{thm:moderatedeviations} are also valid for the exchangeable pair $(-Y,-Y')$, with the same $\lambda$ and the random variable $R$ being replaced by its negative $-R$. This completes the proof of  parts (1) and (2) of Theorem \ref{thm:moderatedeviationsresult}.
	
	We now prove part (3). Once again, let us start with the case $r=0$. To begin with, note that the model \eqref{eq:cw} can be written as:
	\begin{align}\label{cwmod32}
		\p_{\beta,h,p}(\bm X) = \frac{\exp\left\{\beta N\Bar{X}^p + hN\Bar{X}\right\}}{2^N Z_N(\beta,h,p)}.
	\end{align}
	Then for any $m\in \mathcal{M_N} = \{-1, -1+\frac{2}{N},...,1-\frac{2}{N}, 1\}$, 
	$$\p_{\beta,h,p}(\Bar{X}=m)=\frac{1}{2^N Z_N(\beta,h,p)}\binom{N}{\frac{N(1+m)}{2}}e^{N(\beta m^p+hm)}.$$
	Define
	$W_N := N^{-\frac{3}{4}}(S_N - Nm_*)$
	where $m_*$ denotes the unique global maximizer of $H$. Then for $x>0$,
	$$ \p(W_N >x) =\p(S_N> x N^{\frac{3}{4}}+Nm_*).$$
	Now, let us define 
	$$y_{m,N} := 2^{-N}\sqrt{\frac{\pi N(1-m_*^2)}{2}}\binom{N}{\frac{N(1+m)}{2}}e^{N(\beta m^p+hm)-NH(m_*)}. $$
	Then,
	$$ \p(W_N >x) =\frac{1}{\sum\limits_{m\in \mathcal{M}_N}y_{m,N}}\sum\limits_{m\in \mathcal{M}_N}y_{m,N}\bm{1}(m>N^{-\frac{1}{4}}x+m_*).$$
	 Let $\varepsilon :=0.24$ and $\zeta := 3.9$, and define: 
		\begin{align*}
			A_N & = \sum\limits_{m\in \mathcal{M}_N}y_{m,N}\bm{1}\left(|m-m_*|\geq N^{-\frac{1}{4}+\varepsilon}\right)~,\quad B_N = \sum\limits_{m\in \mathcal{M}_N}y_{m,N}\bm{1}\left(|m-m_*|< N^{-\frac{1}{4}+\varepsilon}\right),\\
			\hat{A}_N &= \sum\limits_{m\in \mathcal{M}_N}y_{m,N}\bm{1}\left(m-m_*\geq N^{-\frac{1}{4}+\varepsilon}\right)~,\quad B_{N,x} = \sum\limits_{m\in \mathcal{M}_N}y_{m,N}\bm{1}\left(N^{-\frac{1}{4}}x<m-m_*< N^{-\frac{1}{4}+\varepsilon}\right) . \nonumber 
	\end{align*}  
	In these notations, 
	\begin{align}\label{basicre}
		\begin{aligned}
			\p(W_N>x)= & \frac{\hat{A}_N+B_{N,x}}{A_N+B_N}\\
			= & \frac{\hat{A}_N}{A_N+B_N} + \frac{B_{N,x}}{B_N}\left(1-\frac{A_N}{A_N+B_N}\right).
		\end{aligned}
	\end{align}
	We first estimate $\frac{\hat{A}_N}{A_N+B_N}$ and $\frac{A_N}{A_N+B_N}$.

	\begin{lem}\label{leman}
		  Let $A_N$, $\hat{A}_N$, $B_N$ and $B_{N,x}$ be defined as above. Then, 
			$$
			\frac{\hat{A}_N}{A_N+B_N}\leq \frac{A_N}{A_N+B_N}=\exp\left\{ (1+o(1))\frac{H^{(4)}(m_*)}{4!}N^{4\varepsilon}H^{(4)}(m_*)\right\} O(N^{\frac{3}{2}})
			$$
			for all $N$ large enough. 
	\end{lem}
	
	\begin{proof}
		Clearly, $A_N \leq \hat{A}_N$ and hence, $\frac{\hat{A}_N}{A_N+B_N}\leq\frac{A_N}{A_N+B_N}$. Also, 
		$$ 
		\frac{A_N}{A_N+B_N} = \p\left( |\Bar{X}-m_*|\geq N^{-\frac{1}{4}+\varepsilon}\right) . 
		$$
		The result in Lemma \ref{leman} now follows from \cite[Lemma A.5 ]{mlepaper}. 
	\end{proof}
	
	In the following lemma we obtain estimates for $B_N$ and $B_{N,x}$. The proof is given in Appendix \ref{sec:lembnpf}.  
	
	\begin{lem}\label{lembn}
		 Let $B_N$ and $B_{N,x}$ be as defined above. Then,     
			\begin{align}\label{eq:BNintegral}
				B_N = \frac{N^{\frac{3}{4}}}{2}\int\limits_{-\infty}^{\infty}p_1(t)dt \left(1+O\left(\frac{1}{\sqrt N}\right)\right).
			\end{align}
			Further, there exists a constant $D>1$, such that for $x>D$, 
			\begin{align}\label{eq:bnxdu}
				B_{N,x} = \frac{N^{\frac{3}{4}}}{2}\hat{P}_1(x)+\frac{\sqrt N}{2}\hat{P}_2(x)+O(1)\left(N^{\frac{1}{4}}\hat{R}(x) + p_1(x) + \frac{p_2(x)}{N^{\frac{1}{4}}}+\frac{r(x)}{\sqrt N}   \right) + O\left(e^{-N^{\zeta\varepsilon}}\right),
			\end{align}
			with $\varepsilon = 0.24$ and $\zeta = 3.9$, and for $x\leq D$, 
			\begin{align}\label{eq:bnxd}
				B_{N,x} = \frac{N^{\frac{3}{4}}}{2}\hat{P}_1(x)+\frac{\sqrt N}{2}\hat{P}_2(x)+O(N^{\frac{1}{4}}). 
		\end{align}  
	\end{lem}

	By Lemma \ref{leman}, 
	\begin{align}\label{anbd}
		\frac{\hat{A}_N}{A_N+B_N}\leq \frac{A_N}{A_N+B_N}=\exp\left\{ (1+o(1))\frac{H^{(4)}(m_*)}{4!}N^{4\varepsilon}H^{(4)}(m_*)\right\} O(N^{\frac{3}{2}}).
	\end{align}
	Next, define 
	\begin{equation*}
		\begin{aligned}
			p_1(t) &= e^{\frac{H^{(4)}(m_*)}{4!}t^4} , \\
			p_2(t) &= e^{\frac{H^{(4)}(m_*)}{4!}t^4}\left(\frac{H^{(5)}(m_*)}{5!}t^5 + \frac{m_*}{1-m_*^2} t\right) , \\
			r(t) &= \left(\frac{1}{\sqrt N}(1+t^4) + \frac{t^{11}}{N^{\frac{1}{4}}} + t^2+t^6 + t^{10}\right)e^{\frac{H^{(4)}(m_*)}{4!}t^4}.
		\end{aligned}
	\end{equation*}
	Also, define
	\begin{align*}
		\hat{P}_1(x) = \int\limits_{x}^{\infty}p_1(t)dt,~
		\hat{P}_2(x) = \int\limits_{x}^{\infty}p_2(t)dt,~
		\hat{R}(x) = \int\limits_{x}^{\infty}r(t)dt.
	\end{align*}
	By Lemma \ref{lembn}, there exists a constant $D>1$, such that 
	\begin{align}\label{bsref}
		\frac{B_{N,x}}{B_N} 
		= \frac{\hat{P}_1(x)+N^{-\frac{1}{4}}\hat{P}_2(x)+O(1)\left(\frac{\hat{R}(x)}{\sqrt N} + \frac{p_1(x)}{N^{\frac{3}{4}}} + \frac{p_2(x)}{N} + \frac{r(x)}{N^{ \frac{5}{4} }} \right)}{\hat{P}_1(-\infty) + O\left(\frac{1}{\sqrt N}\right)} + O\left(e^{-N^{\zeta\varepsilon}}\right). 
	\end{align}
	for all $x>D$.
	Since $1-F(x) = \frac{\hat{P}_1(x)}{\hat{P}_1(-\infty)}$, we can write in view of \eqref{basicre}, \eqref{anbd} and \eqref{bsref},
	$$
	\begin{aligned}
		&\frac{\p(W_N>x)}{1-F(x)} \\
		= &\frac{1 + \frac{\hat{P}_2(x)}{N^{\frac{1}{4}} \hat{P}_1(x)}+O(1) \left(  \frac{1}{\sqrt N} \frac{ \hat{R}(x)}{\hat{P}_1(x)}+ \frac{1}{N^{\frac{3}{4}}} \frac{p_1(x)}{\hat{P}_1(x)} + \frac{1}{N} \frac{p_2(x)}{\hat{P}_1(x)} +  \frac{1}{N^{\frac{5}{4}}} \frac{r(x)}{\hat{P}_1(x)}\right)}{1+ O(\frac{1}{\sqrt N})}\left(1-O\left(e^{-N^{\zeta\varepsilon}}\right)\right)\\ +&\frac{O\left(e^{-N^{\zeta\varepsilon}}\right)}{\hat{P}_1(x)}~. 
	\end{aligned}
	$$
	
	
	
	Next, by properties of the incomplete gamma function (see for example \cite{unifgamma}), we have 
	$$ 
	\frac{\hat{R}(x)}{\hat{P}_1(x)} = O(x^{10}),~ \frac{p_1(x)}{\hat{P}_1(x)} = O(x^3),~ \frac{p_2(x)}{\hat{P}_1(x)}=O(x^8),~\text{and}~ \frac{r(x)}{\hat{P}_1(x)} = O(x^{13}),
	$$
	Hence, 
	$$
	\begin{aligned}
		&\frac{\p(W_N>x)}{1-F(x)}\\ & =1+ (1+o(1)) \frac{1}{N^{\frac{1}{4}}} \frac{\hat{P}_2(x)}{\hat{P}_1(x)} + O(1)\left( \frac{x^{10}}{\sqrt N} + \frac{x^3}{N^{\frac{3}{4}}} + \frac{x^8}{N} + \frac{x^{13}}{N^{\frac{5}{4}}} \right) + O(1) x^3 e^{\frac{H^{(4)}(m_*)}{4!} x^4-N^{\zeta\varepsilon}}\\ 
		& = 1+ (1+o(1))\frac{G(x)}{N^{\frac{1}{4}}} + O(1)\left( \frac{x^{10}}{\sqrt{N}} + \frac{x^{13}}{N^{\frac{5}{4}}}\right) + O(1) x^3 e^{\frac{H^{(4)}(m_*)}{4!} x^4-N^{\zeta\varepsilon}}.
	\end{aligned}
	$$
	with $G(x) := \frac{\hat{P}_2(x)}{\hat{P}_1(x)}$.
	
	For $x\leq D$, using \eqref{eq:bnxd}, we have by similar arguments as before, 
	$$
	\frac{\p(W_N>x)}{1-F(x)} = 1+ (1+o(1))\frac{G(x)}{N^{\frac{1}{4}}} + O\left(\frac{1}{\sqrt N} \right).
	$$
	Combining the two cases and the fact that $G(x) = O(1+ x^{5})$, we have
	$$\frac{\p\left(\frac{S_N-Nm_*}{N^{\frac{3}{4}}}>x\right )}{1-F(x)} = 1 + (1+ x^{5}) O\left(\frac{1}{N^{1/4}}\right),$$
	for all $x\in [0,CN^{\frac{1}{20}}]$, thereby completing the proof for the case $r=0$. If $r=1$, the same proof technique works, by replacing $m-m_*$ with its negative in the definitions for $\hat{A}_N$ and $B_{N,x}$. 
	
	\subsection{Proof of Theorem \ref{thm:berryesseen}} 
	\label{sec:convergencepf}
	
	Once again, we prove parts (1) and (2) together. Let $F_N$ and $\overline{F}_N$ denote the cumulative distribution function and the survival function of $\tilde{W}_N$ (recall \eqref{eq:RCWN}), respectively, and let $\overline{\Phi} := 1 - \Phi$ denote the survival function of the standard normal distribution. By Theorem \ref{thm:moderatedeviationsresult} with $r=0$, 
	\begin{align}\label{smallpart}
		\sup_{0\leq x\leq CN^{\frac{1}{6}}} \left|\overline{F}_N(x) - \overline{\Phi}(x)\right| \leq \sup_{t\geq 0} \left\{\overline{\Phi}(t) (1+t^3)\right\} O\left(\frac{1}{\sqrt{N}}\right) = O\left(\frac{1}{\sqrt{N}}\right).
	\end{align}
	For $x > CN^\frac{1}{6}$, we have $\overline{\Phi}(x) \leq \exp(-C^2N^{\frac{1}{3}}/2)$.  Also, since all moments of $\tilde{W}_N$ converge  to the corresponding Gaussian moments (see \cite{smfl}), all moments of $\tilde{W}_N$ are bounded.  Hence, it follows from Markov's inequality, that
	$\overline{F}_N(x) = O(x^{-4}) = O(N^{-\frac{2}{3}})$. Hence,
	\begin{align}\label{largepart}
		\sup_{x> CN^{\frac{1}{6}}} \left|\overline{F}_N(x) - \overline{\Phi}(x)\right|  = O(N^{-\frac{2}{3}}).
	\end{align}
	Combining \eqref{smallpart} and \eqref{largepart}, we conclude that:
	\begin{align}\label{be1}
		\sup_{x\geq 0} \left|\overline{F}_N(x) - \overline{\Phi}(x)\right| = O\left(\frac{1}{\sqrt{N}}\right).
	\end{align}
	
	 Next, note that Theorem \ref{thm:moderatedeviationsresult} with $r=1$ gives, 
		\begin{align}\label{smallpart2}
			\sup_{x \in [0,CN^{\frac{1}{6}}]\bigcap G} \left|F_N(-x) - \Phi(-x)\right| \leq \sup_{t\geq 0} \left\{\overline{\Phi}(t) (1+t^3)\right\} O\left(\frac{1}{\sqrt{N}}\right) = O\left(\frac{1}{\sqrt{N}}\right) , 
		\end{align}
		for all $N\geq 1$, where $G$ is the intersection of the sets of continuity points of $F_n$ for $n \ge 1$, and hence, is the complement of a countable set. By right-continuity of $F_N$ and \eqref{smallpart2}, one has: 
	\begin{align}\label{smallpart3}
		\sup_{x \in [-CN^{\frac{1}{6}},0]} \left|F_N(x) - \Phi(x)\right| = O\left(\frac{1}{\sqrt{N}}\right)
	\end{align}
	Moreover, by a similar argument as before using the tail decay of the Gaussian distribution and Markov's inequality, one also has:
	\begin{align}\label{largepart2}
		\sup_{x < - CN^{\frac{1}{6}}} \left|F_N(x) - \Phi(x)\right|  = O(N^{-\frac{2}{3}}).
	\end{align}
	Combining \eqref{smallpart3} and \eqref{largepart2}, we conclude that:
	\begin{align}\label{be2}
		\sup_{x\leq 0} \left|F_N(x) - \Phi(x)\right| = O\left(\frac{1}{\sqrt{N}}\right).
	\end{align}
	Combining \eqref{be1} and \eqref{be2}, parts of (1) and (2) of Theorem \ref{thm:berryesseen} now follows. 
	
	For part (3), we can similarly let $F_N$ and $\overline{F}_N$ denote the cumulative distribution function and the survival function of $W_N$ (recall \eqref{eq:WN}), respectively, and let $\overline{F}$ denote the survival function of $F$. By Theorem \ref{thm:moderatedeviationsresult}, we have 
	\begin{align}\label{smallpart4}
		\sup_{0\leq x\leq CN^{\frac{1}{20}}} \left|\overline{F}_N(x) - \overline{F}(x)\right| \leq \sup_{t\geq 0} \left\{\overline{F}(t) (1+t^5)\right\} O\left(\frac{1}{N^{1/4}}\right) = O\left(\frac{1}{N^{1/4}}\right). 
	\end{align} 
	For $x > CN^\frac{1}{20}$, we have $\overline{F}(x) = \exp \{ O(N^{-\frac{3}{20} } ) \}$.  Also, since all moments of ${W}_N$ converges to the corresponding moments of $F$ (see \cite{smfl}), all moments of $W_N$ are bounded in this case as well. Hence, it follows from Markov's inequality, that
	$\overline{F}_N(x) = O(x^{-20}) = O(\frac{1}{N})$. Hence,
	\begin{align}\label{largepart3}
		\sup_{x> CN^{\frac{1}{20}}} \left|\overline{F}_N(x) - \overline{F}(x)\right|  = O\left(\frac{1}{N}\right).
	\end{align}
	Combining \eqref{smallpart4} and \eqref{largepart3}, we conclude that:
	\begin{align}
		\sup_{x\geq 0} \left|\overline{F}_N(x) - \overline{F}(x)\right| = O\left(\frac{1}{N^{1/4}}\right). \nonumber 
	\end{align}
	This establishes the result in part (3) and completes the proof of Theorem \ref{thm:berryesseen}.
	
	\subsection{Proof of Theorem \ref{thm:berryesseenestimate}} 
	\label{sec:estimationpf}
	
	In this section $\bar{\p}$ will be used to abbreviate the conditional distribution $\p_{\beta,0,p}(\cdot | \bs \in A)$. Also, define $$a:= g'(m)  \quad \text{ and } \quad b:= \frac{g''(m)}{2\sqrt{N}},$$ where $g(t) := p^{-1}t^{1-p}\tanh^{-1}(t)$.

	\begin{lem}\label{aprxm1}
		Suppose that $m$ is a global maximizer of $H_{\beta,0,p}$ and let $A$ be a neighborhood of $m$ whose closure is devoid of any other maximizer. 
		Let $W = \sqrt{N}(\bs - m)$. Then, for every $t\in \mathbb{R}$, 
		\begin{align*}
			\bar{\p}\left(aW + bW^2 > t+\frac{1}{\sqrt{N}}\right) - O\left(\frac{1}{N}\right) & \leq \bar{\p}\left(\sqrt{N}(\hat{\beta} - \beta)>t\right)  \nonumber \\ 
			& \leq \bar{\p}\left(aW + bW^2 > t-\frac{1}{\sqrt{N}}\right) +O\left(\frac{1}{N}\right) ,
		\end{align*}
		where $\hat \beta$ is the MPL estimate of $\beta$ in the model $\p_{\beta,0,p}$. 
	\end{lem}
	
	\begin{proof}
		First, note that by a third-order Taylor expansion, 
		$$\sqrt{N}(\hat{\beta}-\beta) = \sqrt{N}(g(\bs)-g(m)) = aW+ bW^2 + Q$$ where $Q := \sqrt{N}(\bs - m)^3 g^{(3)}(\xi)/6$ for some random variable $\xi$ between $\bs$ and $m$. Hence, 
		$$\bar{\p}\left(\sqrt{N}(\hat{\beta} - \beta)>t\right) \leq \bar{\p}\left(aW + bW^2 > t-\frac{1}{\sqrt{N}}\right) + \bar{\p}\left(Q>\frac{1}{\sqrt{N}}\right)$$
		and
		$$\bar{\p}\left(\sqrt{N}(\hat{\beta} - \beta)>t\right) \geq \bar{\p}\left(aW + bW^2 > t+\frac{1}{\sqrt{N}}\right) - \bar{\p}\left(Q\leq -\frac{1}{\sqrt{N}}\right).$$
		Thus, to prove Lemma \ref{aprxm1} it suffices to show that
		\begin{align*}
			\bar{P}\left(|Q| \geq \frac{1}{\sqrt{N}}\right) = O\left(\frac{1}{N}\right).
		\end{align*}
		This follows from the fact that the moments of $W$ are bounded, so by Markov's inequality, one has:
		$$\bar{\p}\left(|Q| \geq \frac{1}{\sqrt{N}}\right) \leq N \e_{\bar{\p}} (Q^2) = O\left(N^2 \e_{\bar{\p}}(\bs-m)^6\right) = \frac{1}{N} O(\e (W^6)) = O\left(\frac{1}{N}\right).$$
		This completes the proof of Lemma \ref{aprxm1}.
	\end{proof}
	
	In the next lemma, we show how to derive Berry-Esseen type bounds for the random variable $aW+ bW^2$ from the Berry-Esseen bounds for $W$ established in Theorem \ref{thm:berryesseen}.
	
	\begin{lem}\label{beressquad}
		Recall the notations of Lemma \ref{aprxm1} and let $Z\sim N(0,-1/H''(m))$. Then, 
		$$\sup_{t\in \mathbb{R}} \left|\bar{\p}(a W + b W^2> t) -     \p(a Z + >t)\right| = O\left(\frac{\log N}{\sqrt{N}}\right) , $$ 
		where the constant in the $O(\cdot)$ term depends on $\beta$ and $p$.
	\end{lem}
	\begin{proof}
		Denote $\psi(t) := at + bt^2$. Note that the set $J:= \psi^{-1}((t,\infty))$ is a union of two (possibly empty) open intervals, i.e., there exist $a_1(t)\leq b_1(t) < a_2(t)\leq b_2(t)$ in the extended real line, such that $J = (a_1(t),b_1(t))\bigcup (a_2(t),b_2(t))$.  Now, note that:
			\begin{align}\label{eq:WZ}
				&\left|\bar{\p}(a W + b W^2> t) -     \p(a Z + bZ^2>t)\right| \nonumber \\ 
				&= \left|\bar{\p}(W \in J) -     \p(Z\in J)\right| \nonumber \\
				&\leq \sum_{i=1}^2 \left|\bar{\p}(W<b_i(t))-\p(Z<b_i(t))\right| + \sum_{i=1}^2 \left|\bar{\p}(W\leq a_i(t))-\p(Z\leq a_i(t))\right|  \nonumber \\ 
				& \leq \frac{C}{\sqrt N} , 
			\end{align} 
			where the last step uses Theorem \ref{thm:berryesseen}, for some universal constant $C>0$.  
		
		%
		%
		%
		
		Now, define $\alpha_N := \sqrt{-\log N/H''(m)}$. Note that 
		$$\p(|Z|>\alpha_N) \leq 2\exp\left(\frac{\alpha_N^2 H''(m)}{2}\right) = \frac{2}{\sqrt{N}}~.$$
		Hence, 
		\begin{align}\label{st41}
			&\left|\p(aZ+bZ^2 > x) - \p(aZ>x) \right|\nonumber\\ 
			&\leq   \left|\p(aZ+bZ^2 > x, |Z|\leq \alpha_N) - \p(aZ>x, |Z|\leq \alpha_N)\right| + 2\p(|Z|>\alpha_N)\nonumber\\
			&\leq \left|\p(aZ+bZ^2 > x, |Z|\leq \alpha_N) - \p(aZ>x, |Z|\leq \alpha_N)\right| + \frac{4}{\sqrt{N}} . 
		\end{align}
		Next, note that:
		$$\p(aZ>x + |b|\alpha_N^2, |Z|\leq \alpha_N)~ \le~ \p(aZ+bZ^2 > x, |Z|\leq \alpha_N) ~\le~ \p(aZ>x - |b|\alpha_N^2, |Z|\leq \alpha_N).$$
		 Also, since the density function of $Z$ is bounded, 
			$$\max_{r\in \{-1,1\}} \left|\p(aZ>x + (-1)^r|b|\alpha_N^2, |Z|\leq \alpha_N) - \p(aZ>x , |Z|\leq \alpha_N)\right| = O\left(|b|\alpha_N^2\right) = O\left(\frac{\log N}{\sqrt{N}}\right) . $$ 
		Hence, 
		\begin{align}\label{st42}
			\left|\p(aZ+bZ^2 > x, |Z|\leq \alpha_N) - \p(aZ>x , |Z|\leq \alpha_N)\right| = O\left(\frac{\log N}{\sqrt{N}}\right) . 
		\end{align}
		Combining \eqref{st41} and \eqref{st42} gives, 
		\begin{align}\label{eq:Z}
			\sup_{x\in \mathbb{R}}~\left|\p(aZ+bZ^2 > x) - \p(aZ>x) \right| = O\left(\frac{\log N}{\sqrt{N}}\right) . 
		\end{align}
		The result in Lemma \ref{beressquad} now follows from \eqref{eq:WZ} and \eqref{eq:Z}. 
	\end{proof}

	In view of Lemma \ref{aprxm1} and Lemma \ref{beressquad} we have the following for every $t\in \mathbb{R}$ and every global maximizer $m$ of $H_{\beta,0,p}$:
	$$ \p\left(aZ > t +\frac{1}{\sqrt{N}}\right) - O\left(\frac{\log N}{\sqrt{N}}\right) \leq  \p\left(\sqrt{N}(\hat{\beta}-\beta)>t\Big| \bar{X}\in A\right) 
	\leq \p\left(aZ > t -\frac{1}{\sqrt{N}}\right) + O\left(\frac{\log N}{\sqrt{N}}\right) , $$ 
	where $A$ is a neighborhood of $m$ whose closure is devoid of any other global maximizer, $a := g'(m)$ with $g(x) := p^{-1}x^{1-p}\tanh^{-1}(x)$, $Z\sim N(0,-1/H''(m))$, and the hidden constants in the $O(\cdot)$ terms depend only on $\beta$ and $p$. 
	Since the density of $aZ$ is bounded, one also has:
	$$\sup_{t\in \mathbb{R}} ~ \left|\p\left(aZ>t\pm \frac{1}{\sqrt{N}}\right) - \p(aZ>t) \right| = O\left(\frac{1}{\sqrt{N}}\right).$$
	 This shows that 
		\begin{align}\label{bepnt}
			\sup_{t\in \mathbb{R}}~ \left|\p\left(\sqrt{N}(\hat{\beta}-\beta)>t\Big| \bar{X}\in A\right)~-~\p(aZ>t)\right|~=~O\left(\frac{\log N}{\sqrt{N}}\right).
		\end{align}
		Note that $aZ \sim N(0,-(g'(m))^2/H''(m))$. The result in Theorem \ref{thm:berryesseenestimate} now follows from \eqref{bepnt} on observing that $-\frac{(g'(m))^2}{H''(m)} = -\frac{H''(m)}{p^2m^{2p-2}}$, and on noting that $\p(\bar{X}\in A^c) $ is exponentially small (see Lemma 3.1 in \cite{mlepaper}).   
	\qed

	\section{Future Directions} 
	\label{sec:extensions}
	
	There are several directions in which the results in this paper can be extended. For instance, one can consider the $p$-spin Potts model, for which the fluctuations of the magnetization has been proved recently by 
	Bhowal and Mukherjee \cite{bmcurieweisspotts}. Obtaining rates of convergence for these results would be a natural next direction. Recently, the fluctuations of the magnetization in the 2-Curie-Weiss model were shown to be universal, that is, they continue to have the same asymptotic distribution for approximately regular graphs \cite{meanfieldising}. Extending this result to the $p$-spin case is another possible future direction. 
	
	Going beyond the classical setting, it might also be interesting to study the quantum Curie-Weiss model \cite{chayes2008phase} and possible 
	connections with Parthasarathy's groundbreaking work on quantum probability  \cite{stochasticcalculus,qm}.

	\appendix
	
	%
	%

	
	
	\section{Proof of Lemma \ref{lembn} }
	\label{sec:lembnpf} 
	
	\begin{proof}
		
		We begin by recalling a result about the properties of the higher-order derivatives of the function $H$, which follows from Lemma B.2 of \cite{smfl}.
		
		\begin{lem}\label{Hderi}
			Let $H(x)$ be defined as in \eqref{Hdefn}. Suppose that $(\beta,h) \in \cS_p$ and $m_*$ is the unique maximizer of $H$. Then, 
			$$
			H'(m_*)=H''(m_*)=H^{(3)}(m_*)=0\quad\text{and}\quad H^{(4)}(m_*) <0.
			$$
		\end{lem}


		To begin with, suppose  that $|m-m_*|\leq N^{-\frac{1}{4}+\varepsilon}$. By Lemma \ref{Hderi} and a Taylor expansion of $H(m)$ around $m_*$, 
		$$H(m)-H(m_*)= \frac{H^{(4)}(m_*)}{4!}(m-m_*)^4 + \frac{H^{(5)}(m_*)}{5!}(m-m_*)^5+ \frac{H^{(6)}(\xi_{m,N})}{6!}(m-m_*)^6,$$
		where $H^{(5)}$ and $H^{(6)}$ denotes the fifth and sixth derivatives of ythe function $H$, respectively, and $\xi_{m,N}$ lies between $m$ and $m_*$. For $|m-m_*|\leq N^{-\frac{1}{4}+\varepsilon}$, $N(m-m_*)^5$ and $N(m-m_*)^6$ are both $o(1)$. Noticing that $e^x = 1+ x+ O(x^2)$ for $x\rightarrow 0$, we have: 
		\begin{align*}
			& e^{N[H(m)-H(m_*)]}\\ 
			&= e^{N\frac{H^{(4)}(m_*)}{4!}(m-m_*)^4}\left(1+\frac{H^{(5)}(m_*)}{5!}N(m-m_*)^5+O\left(N^2(m-m_*)^{10}\right)\right)\left(1 + O\left(N(m-m_*\right)^6)\right)\\
			&= e^{\frac{H^{(4)}(m_*)}{4!}\frac{(s_N-N m_*)^4}{N^3}}\left(1+\frac{H^{(5)}(m_*)}{5!}\frac{(s_N-N m_*)^5}{N^4}+O(1)M_{s_N, N}\right) , 
		\end{align*}
		where $$s_N := Nm \quad \text{ and } \quad M_{s_N, N}=\frac{(s_N-N m_*)^6}{N^5}+ \frac{(s_N-N m_*)^{10}}{N^8} + \frac{(s_N-N m_*)^{11}}{N^9}.$$
		Now, from Lemma A.5 in \cite{smfl}, 
		\begin{align}\label{bdymn}
			y_{m,N}= \left(1+O \left( \frac{1}{N} \right) \right) \sqrt{\frac{1-m_*^2}{1-m^2}}e^{N[H(m)-H(m_*)]} , 
		\end{align}
		whenever $m$ is bounded away from $\pm 1$.
		Also, now follows from \eqref{bdymn} that for $|m-m_*|<N^{-\frac{1}{4}+\varepsilon}$,
		\begin{align}\label{bdymnh}
			y_{m,N} = \sqrt{\frac{1-m_*^2}{1-m^2}} e^{\frac{H^{(4)}(m_*)}{4!}\frac{(s_N-N m_*)^4}{N^3}}\Bigl( 1+ \frac{H^{(5)}(m_*)}{5!}\frac{(s_N-N m_*)^5}{N^4} + O(1)R_{s_N, N}\Bigr), 
		\end{align}
		where $$R_{s_N, N} = \frac{1}{N}+\frac{(S_N-Nm_*)^5}{N^5}+\frac{(S_N-N m_*)^6}{N^5}+ \frac{(S_N-N m_*)^{10}}{N^8} + \frac{(S_N-N m_*)^{11}}{N^9}.$$
		
		Next, note that 
		$$\sqrt{\frac{1}{1-m^2}} = \sqrt{\frac{1}{1-m_*^2}} + m_*(1-m_*^2)^{-\frac{3}{2}}\frac{s_N-Nm_*}{N} + O\left(\frac{(s_N-Nm_*)^2}{N^2}\right) . 
		$$
		Hence,
		$$\sqrt{\frac{1-m_*^2}{1-m^2}} = 1+ \frac{m_*}{1-m_*^2}\cdot\frac{s_N-Nm_*}{N} + O\left(\frac{(s_N-Nm_*)^2}{N^2}\right)~.$$
		Therefore, from \eqref{bdymnh}, 
		$$y_{m,N} =  e^{\frac{H^{(4)}(m_*)}{4!}\frac{(s_N-N m_*)^4}{N^3}}\Bigl( 1+ \frac{H^{(5)}(m_*)}{5!}\frac{(s_N-N m_*)^5}{N^4} + 
		\frac{m_*}{1-m_*^2}\cdot\frac{s_N-Nm_*}{N} +  O(1)U_{s_N, N}\Bigr),$$
		where 
		$$U_{s_N, N} = \frac{1}{N}+ \frac{(s_N-Nm_*)^2}{N^2} + \frac{(s_N-Nm_*)^5}{N^5}+\frac{(s_N-N m_*)^6}{N^5}+ \frac{(s_N-N m_*)^{10}}{N^8} + \frac{(s_N-N m_*)^{11}}{N^9}.$$
		
		 Defining the auxiliary variable $\ell := s_N - Nm_*$, we have
			\begin{align}\label{bn}
				B_N & = \sum\limits_{|m-m_*|< N^{-\frac{1}{4}+\varepsilon}, ~ m\in \mathcal{M}_N}y_{m,N}\\
				& = \sum\limits_{\substack{|\ell | < N^{\frac{3}{4}+\varepsilon} \\ 2| ( \ell + Nm_* + N)}} e^{\frac{H^{(4)}(m_*) \ell^4 }{4! N^3} } \left( 1+ \frac{H^{(5)}(m_*) }{5!}\frac{\ell^5}{N^4} + \frac{m_*}{1-m_*^2}\cdot\frac{\ell}{N} + O\left( R_{\ell, N} \right) \right) , \nonumber 
			\end{align}
			where 
			\begin{align}    
				R_{\ell, N} & := \frac{1}{N}+\frac{\ell^2}{N^2} +\frac{\ell^5}{N^5} + \frac{\ell^6}{N^5} + \frac{\ell^{10}}{N^8} + \frac{\ell^{11}}{N^9}  \nonumber \\ 
				& =   \frac{O(1)}{\sqrt N}\left( \frac{1}{\sqrt N}+\left(\frac{\ell}{N^{\frac{3}{4}}}\right)^2  + \frac{1}{N^{\frac{3}{4}}}\left(\frac{\ell}{N^{\frac{3}{4}}}\right)^5+\left(\frac{\ell}{N^{\frac{3}{4}}}\right)^6+\left(\frac{\ell}{N^{\frac{3}{4}}}\right)^{10}+ \frac{1}{N^{\frac{1}{4}}}\left(\frac{\ell}{N^{\frac{3}{4}}}\right)^{11} \right) . \nonumber 
		\end{align} 
		Therefore, from \eqref{bn}
		\begin{align}\label{eq:prbn}
			B_N & = \sum\limits_{\substack{| \ell | <N^{\frac{3}{4}+\varepsilon} \\ 2| ( \ell + Nm_* + N)}}p_1\left(\frac{\ell}{N^{\frac{3}{4}}}\right) + \frac{1}{N^{\frac{1}{4}}}\sum\limits_{\substack{| \ell | <N^{\frac{3}{4}+\varepsilon} \\ 2| ( \ell + Nm_* + N)}}p_2\left(\frac{\ell}{N^{\frac{3}{4}}}\right)+ \frac{O(1)}{\sqrt N}\sum\limits_{\substack{| \ell | <N^{\frac{3}{4}+\varepsilon} \\ 2| ( \ell + Nm_* + N)}}r\left(\frac{\ell}{N^{\frac{3}{4}}}\right), 
		\end{align}
		where
		$$
		\begin{aligned}
			p_1(t) &= e^{\frac{H^{(4)}(m_*)}{4!}t^4} , \\
			p_2(t) &=  e^{\frac{H^{(4)}(m_*)}{4!}t^4}\left(\frac{H^{(5)}(m_*)}{5!}t^5 + \frac{m_*}{1-m_*^2} t\right), \\
			r(t) &= \left(\frac{1}{\sqrt N} + t^2 + \frac{t^5}{N^{\frac{3}{4}} }  + t^6 + t^{10} + \frac{t^{11}}{N^{ \frac{1}{4}} } \right)e^{\frac{H^{(4)}(m_*)}{4!}t^4} . 
		\end{aligned}
		$$
		Using Lemma 2.1 (ii) in \cite{vanh} gives, 
		\begin{align}\label{eq:pt}
			\sum\limits_{\substack{|\ell |< N^{\frac{3}{4}+\varepsilon}\\ 2| (\ell+Nm_* + N)}}p\left(\frac{\ell}{N^{\frac{3}{4}}}\right) = \frac{N^{\frac{3}{4}}}{2}\int\limits_{-N^{\varepsilon}}^{N^{\varepsilon}}p(t)dt + O(N^{\varepsilon})\quad\quad \text{for}\quad p=p_1,p_2~\text{and}~r. 
		\end{align}
		Moreover, by Markov's inequality one has, 
		\begin{align}\label{eq:prt}
			\int\limits_{|t|\geq N^{\varepsilon}}p(t)dt = O\left(e^{-N^{3\varepsilon}}\right)\quad\quad \text{for}\quad p=p_1,p_2~\text{and}~r. 
		\end{align}
		Combining the bounds in \eqref{eq:pt} and \eqref{eq:prt} we get, 
		\begin{align}\label{eq:p}
			 \sum\limits_{\substack{|\ell |< N^{\frac{3}{4}+\varepsilon}\\ 2| (\ell+Nm_* + N)}}p\left(\frac{\ell}{N^{\frac{3}{4}}}\right) = \frac{N^{\frac{3}{4}}}{2}\int\limits_{-\infty}^{\infty}p(t)dt + O\left(N^{\varepsilon}\right)\quad\quad \text{for}\quad p=p_1,p_2~\text{and}~r.  
		\end{align}
		Combining \eqref{eq:prbn} with \eqref{eq:p}, we get: 
		\begin{align*}
			B_N = \frac{N^{\frac{3}{4}}}{2}\int\limits_{-\infty}^{\infty}p_1(t)dt +\frac{\sqrt N}{2}\int\limits_{-\infty}^{\infty}p_2(t)dt + O\left(N^{\frac{1}{4}}\right) = \frac{N^{\frac{3}{4}}}{2}\int\limits_{-\infty}^{\infty}p_1(t)dt\left[1+ O\left(\frac{1}{\sqrt N}\right)\right]. 
		\end{align*}  
		This completes the proof of \eqref{eq:BNintegral}.

		We now estimate $B_{N,x}$. With similar argument as in \eqref{bn}, we have
		\begin{equation*}
			\begin{aligned}
				B_{N,x} =& \sum\limits_{m\in \mathcal{M}_N}y_{m,N}\bm{1}\{ N^{-\frac{1}{4}}x<m-m_*<N^{-\frac{1}{4}+\varepsilon} \}\\
				= & \sum\limits_{\substack{N^{\frac{3}{4}}x < \ell < N^{\frac{3}{4}+\varepsilon} \\ 2| (\ell+Nm_* + N)}}p_1\left(\frac{\ell}{N^{\frac{3}{4}}}\right) + \frac{1}{N^{\frac{1}{4}}}\sum\limits_{\substack{N^{\frac{3}{4}}x < \ell < N^{\frac{3}{4}+\varepsilon} \\ 2| ( \ell + Nm_* + N)}}p_2\left(\frac{\ell}{N^{\frac{3}{4}}}\right)+ \frac{O(1)}{\sqrt N}\sum\limits_{\substack{N^{\frac{3}{4}}x < \ell < N^{\frac{3}{4}+\varepsilon} \\ 2| ( \ell + Nm_* + N)}}r\left(\frac{\ell}{N^{\frac{3}{4}}}\right).
			\end{aligned}
		\end{equation*}
		Clearly, there exists a constant $D>1$ such that $p_1$, $p_2$ and $r$ are decreasing on $[D, \infty)$. Applying Lemma 2.1 (i) in \cite{vanh}, we have for $x \geq D$,
		$$
		\begin{aligned}
			& B_{N,x} = \frac{N^{\frac{3}{4}}}{2}\int\limits_{x}^{N^{\varepsilon}}p_1(t)dt +\frac{\sqrt N}{2}\int\limits_{x}^{N^{\varepsilon}}p_2(t)dt + \\
			& \ \ \ \ \ \ O(1)\left( \frac{N^{\frac{1}{4}}}{2}\int\limits_{x}^{N^{\varepsilon}}r(t)dt + p_1(x) + \frac{p_2(x)}{N^{\frac{1}{4}}}+\frac{r(x)}{\sqrt N} + p_1(N^{\varepsilon})+\frac{p_2(N^{\varepsilon})}{N^{\frac{1}{4}}} + \frac{r(N^{\varepsilon})}{\sqrt N}\right),
		\end{aligned}
		$$
		 Moreover, 
			$$\max \left\{ p_1(N^{\varepsilon}),p_2(N^{\varepsilon}), r(N^{\varepsilon}), \int\limits_{N^{\varepsilon}}^{\infty}p_1(t)dt, \int\limits_{N^{\varepsilon}}^{\infty}p_2(t)dt, \int\limits_{N^{\varepsilon}}^{\infty}r(t)dt \right\} = O\left(e^{-N^{\zeta\varepsilon}}\right)$$ 
			where $\varepsilon = 0.24$ and $\zeta = 3.9$.   
		Then we get, 
		\begin{align*}
			B_{N,x} = \frac{N^{\frac{3}{4}}}{2}\hat{P}_1(x)+\frac{\sqrt N}{2}\hat{P}_2(x)+O(1)\left(N^{\frac{1}{4}}\hat{R}(x) + p_1(x) + \frac{p_2(x)}{N^{\frac{1}{4}}}+\frac{r(x)}{\sqrt N} \right) + O\left(e^{-N^{\zeta\varepsilon}}\right) . 
		\end{align*}
		This proves \eqref{eq:bnxdu}. 
		For $x\leq D$, we can use Lemma 2.2 (ii) in \cite{vanh} to get, 
		\begin{align*}
			B_{N,x} = \frac{N^{\frac{3}{4}}}{2}\hat{P}_1(x)+\frac{\sqrt N}{2}\hat{P}_2(x)+O\left(N^{\frac{1}{4}}\right).
		\end{align*} 
		This proves \eqref{eq:bnxd}, and completes the proof of Lemma \ref{lembn}.  
	\end{proof}

\end{document}